\documentclass[12pt]{amsart}

\usepackage[english]{babel}
\usepackage[latin1]{inputenc}
\usepackage{amsmath,amssymb, 
enumerate,amssymb,wasysym,mathrsfs,dsfont,tabularx}
 \usepackage{mathptmx}
 \usepackage[all,cmtip]{xy}

 \DeclareMathAlphabet{\mathcal}{OMS}{cmsy}{m}{n}

\DeclareSymbolFont{rsfsript}{OMS}{rsfs}{m}{n}
\DeclareSymbolFontAlphabet{\mathrsfs}{rsfsript}

\DeclareSymbolFont{AMSb}{U}{msb}{m}{n}
\DeclareSymbolFontAlphabet{\mathbb}{AMSb}

\DeclareSymbolFont{eufrak}{U}{euf}{m}{n}
\DeclareSymbolFontAlphabet{\gothic}{eufrak}
\DeclareMathOperator{\op}{op}

\newcommand{\Dg}{\mathsf{DiGp}}

\newcommand{\Set}{\mathsf{Set}}
\newcommand{\DHp}{\mathsf{DHeap}}
\newcommand{\Sym}{\operatorname{Sym}}
\newcommand{\Aut}{\operatorname{Aut}}

\DeclareMathOperator{\E}{E}
\DeclareMathOperator{\U}{U}
\DeclareMathOperator{\I}{I}

\newcommand{\Grp}{\mathsf{Grp}}

\DeclareMathOperator{\Triv}{\bf Triv}

\newcommand{\id}{\operatorname{id}}
\newcommand{\N}{\mathbb N}

\newtheorem{theorem}{Theorem}[section]
\newtheorem{proposition}[theorem]{Proposition}
\newtheorem{lemma}[theorem]{Lemma}

\newtheorem{definition}[theorem]{Definition}
\newtheorem{remark}[theorem]{Remark}

\newtheorem{Ex}[theorem]{Example}
\newtheorem{remarks}[theorem]{Remarks}

\newcommand{\SGrp}{\mathsf{SGrp}}
\newcommand{\DSGp}{\mathsf{DSGp}}
\newcommand{\Hom}{\operatorname{Hom}}

\begin{document}
  \title[Digroups, their canonical pretorsion theory, and diheaps]{Digroups, their canonical pretorsion theory, \\ \protect and diheaps}
     \author{Alberto Facchini}
\address[Alberto Facchini]{Dipartimento di Matematica ``Tullio Levi-Civita'',\linebreak Universit\`a di 
Padova, 35121 Padova, Italy}
 \email{facchini@math.unipd.it}
\thanks{The second author is partially supported by GNSAGA,  the projects PIACERI ``PLGAVA-Propriet\`a locali e globali di anelli e di variet\`a algebriche'' and ``MTTAI - Metodi topologici in teoria degli anelli e loro ideali'' of University of Catania, and  the research project PRIN ``Squarefree Gr\"obner degenerations, special varieties and related topics''.} 
\author{Carmelo Antonio Finocchiaro}
\address[Carmelo Antonio Finocchiaro]{Dipartimento di Matematica e Informatica, Universit\`a\ di Catania, Citt\`a\ Universitaria, 95125 Catania, Italy}
\thanks{}
\email{cafinocchiaro@unict.it}


      \begin{abstract} In the category of digroups  there is a natural pretorsion theory in which the torsion-free digroups are all groups, and torsion digroups form a category isomorphic to the category of non-empty sets. It is also possible to extend the theory of heaps from groups to digroups. The corresponding notion is that of a diheap.\end{abstract}

    \maketitle


\section{Introduction} 

Digroups were introduced by Felipe  \cite{Phillips} in 2006 (also see \cite{Kinyon}). The idea is replacing the operation of a group with two operations, a left multiplication $\vdash$ and a right multiplication $\dashv$. Something similar was done by the first author for skew braces and trusses in \cite{Fac, XTrusses}.

In this work we discuss some aspects of the theory of digroups, focusing in particular on their natural pretorsion theory and on the notion corresponding to that of a heap, namely the notion of a diheap.
A digroup is an algebra $(D,\vdash,\dashv)$, where $(D,\vdash)$ is a right group, $(D,\dashv)$ is a left group, $x\vdash (y\dashv z)=(x\vdash y)\dashv z$ for every $x,y,z\in D$, and there exists an element $e\in D$ which is idempotent (both in $(D,\vdash)$ and in $(D,\dashv)$) and is central in the sense that $x\vdash e=e\dashv x$ for every $x\in D$. Denoting by $\U(D)$ the set of idempotents of $D$, often called {\em bar-units} of $D$, and by $\I(D)$ the set of central idempotents of $D$, we have that a disemigroup $D$ is a digroup if and only if $\I(D)$ is a non-empty set. In a digroup, a fundamental role is played by a congruence $\sim$, which can be equivalently defined as the equivalence relation such that $x\sim y $ if $x \vdash e=y \vdash e$ for any idempotent element $e$ of $(D,\vdash)$, or as the equivalence relation such that $x\sim y$ if $e \dashv x=e \dashv y$ for any idempotent element $e$ of $(D,\dashv)$, or as the congruence on $(D,\vdash,\dashv)$ generated by the subset $\{\,(x\vdash y, x\dashv y)\mid x,y\in D\,\}$ of $D\times D$. Certain idempotent endomorphisms of the digroup $(D,\vdash,\dashv)$ play a special role. These are those of the form $\sigma_e\colon D\to D$, where $\sigma_e(x)=x\vdash e$ for every $x\in D$ and every central idempotent $e$ of $D$. There exists a pretorsion theory in the category of digroups where the torsion-free digroups are those of the form $(G,\cdot,\cdot)$ with $(G,\cdot)$  being a group, and where the torsion digroups are all and only those of the form $(X,\vdash,\dashv)$, where $X$ is a non-empty set and $\vdash,\dashv$ are defined by $x\vdash y=y$ and $x\dashv y=x$ for every $x,y\in X$. The preexact sequence corresponding to any digroup $D$
is the sequence $\U(D)\to D\to D/{\sim}$.

Finally we show that the notion of a heap, a ``group with the identity element forgotten'', can be adapted to digroups, for which we find the notion a diheap, an algebra with two ternary operations. 

\section{Preliminary notions on disemigroups and right groups}

\subsection{Disemigroups}
\begin{definition}\label{disemigroup}{\rm A {\em disemigroup} $(D, \vdash,\dashv)$ is a set $D$ together with two binary operations $ \vdash$ and $\dashv$ such that:

(a) $(D,\vdash)$ and $(D,\dashv)$ are semigroups; 

(b) $x\vdash (y\dashv z)=(x\vdash y)\dashv z$;

(c) $x\dashv (y\vdash z)=x\dashv (y\dashv z)$;

(d) $(x\dashv y)\vdash z=(x\vdash y)\vdash z$

\noindent for every $x,y,z\in D$.}\end{definition}

We will call the operation $\vdash$ {\em the left multiplication} and the operation $\dashv$ {\em the right multiplication}  of the disemigroup $D$.

\bigskip

Disemigroup morphisms are the mappings that preserve both operations $ \vdash$ and $\dashv$. 
\begin{definition}{\rm An element $e$ in a disemigroup $D$ is a {\em bar-unit} if $e\vdash x=x\dashv e=x$ for all $x\in D$. 
 We will denote by $\U(D,\vdash,\dashv)$ the set of all bar-units of the disemigroup $D$ (in case there is no danger of confusion we will simply write $\U(D)$). The set $\U(D)$ is sometimes called the {\em halo} of the disemigroup~$D$.}\end{definition}

 For any semigroup $(S,\cdot)$, let $\E(S,\cdot)$ denote the set of all idempotent elements $(S,\cdot)$. Hence, is a disemigroup $(D,\vdash,\dashv)$, we have that $$\U(D,\vdash,\dashv)\subseteq\E(D,\vdash)\cap\E(D,\dashv).$$
\begin{remark}\label{simmetria}{\rm There is an involutory automorphism of the category of disemigroups that associates to each disemigroup $(D, \vdash,\dashv)$ the disemigroup $(D, \dashv^{\op},\vdash^{\op})$ and is the identity on morphisms.}
\end{remark}

The first example of disemigroup is, for any semigroup $(S,\cdot)$, the disemigroup $(S,\cdot,\cdot)$. More generally, we have the following straightforward result:

\begin{lemma}\label{2.4} There is functor $F\colon {\SGrp}\to\DSGp$ of the category ${\SGrp}$ of semigroups into the category $\DSGp$ of disemigroups that associates to every semigroup $(S,\cdot)$ the disemigroup $(S,\cdot,\cdot)$. It is full and faithful, and induces a canonical category isomorphism of the category ${\SGrp}$ of semigroups into the full subcategory of $\DSGp$ whose objects are the disemigroups $(D, \vdash,\dashv)$ for which the two operations $\vdash$ and $\dashv$ coincide (that is, $x\vdash y=x\dashv y$ for every $x,y\in D$).\end{lemma}

In view of the previous lemma, we will call {\em semigroups}  the disemigroups $(D, \vdash,\dashv)$ for which the two operations $\vdash$ and $\dashv$ coincide.

\begin{definition} {\rm Let $(D, \vdash,\dashv)$ be a disemigroup. For a given bar-unit $e\in D$ and any element $x\in D$, a {\em simultaneous inverse} of $x$ with respect to $e$ is an element $x^{\dagger_e}\in D$ such that $x\vdash x^{\dagger_e}=x^{\dagger_e} \dashv x=e$. 
We will denote by $\I(D,\vdash,\dashv)$ (or by $\I(D)$ when there is no danger of confusion) the subset of $\U(D,\vdash,\dashv)$ consisting of all the bar-units $e$ such that any element of $D$ admits a simultaneous inverse with respect to $e$. }	\end{definition}

Therefore, in a given disemigroup $(D,\vdash,\dashv)$, we have that $$\I(D,\vdash,\dashv)\subseteq\U(D,\vdash,\dashv)\subseteq\E(D,\vdash)\cap\E(D,\dashv).$$

\subsection{Right groups} A semigroup $(S,\cdot)$ is a {\em right zero semigroup} \cite[p.~4]{CPI} if $a\cdot b=b$ for all $a,b\in S$. For these semigroups we will usually write the operation $\cdot$ as $\pi_2$, because it corresponds to the second canonical projection $\pi_2\colon S\times S\to S$. Thus $a\,\pi_2\, b=b$. Similarly, $a\,\pi_1\, b=~a$. Hence right zero semigroups are those of the form $(S,\pi_2)$ for some set $S$. The full subcategory of the category of semigroups whose objects are all right zero semigroups is clearly isomorphic to the category $\Set$ of sets, because every mapping between two right zero semigroups is a semigroup morphism.

\bigskip

For every semigroup~$(S,\cdot)$
it is possible to define a mapping $$\ell\colon  S\to S^S,\qquad s \mapsto \ell_s,$$  where 
$\ell_s(t)=s\cdot t$ for every $t\in S$. This mapping  $\ell\colon  S\to S^S$ is a semigroup morphism. The semigroup $(S,\cdot)$ is  {\em left cancellative} if all the mappings $\ell_s\colon S\to S$ ($s\in S$) are injective; \emph{right simple} if all the mappings $\ell_s\colon S\to S$ are surjective;
a {\em right group} if all the mappings $\ell_s\colon S\to S$ are bijective. If $(S,\cdot)$ is a right group, all the mappings $\ell_s\colon S\to S$ have inverse mappings  $(\ell_s)^{-1}\colon S\to S$, hence it is possible to define another binary operation $\backslash $ on the set $S$ setting $s\backslash t := (\ell_s)^{-1}(t)$ for all $s,t\in S$. We will say that the operation $\backslash$ is the {\em left  inverse} of the operation $\cdot$ of the right group $S$. Since $\ell_s$ and $(\ell_s)^{-1}$ are mutually inverse mappings $S\to S$, it follows that $s\cdot (s\backslash t) = t = s\backslash (s\cdot t)$ for all $s,t\in S$. Conversely, let $S$ be a nonempty set endowed with two binary operations $\cdot, \backslash$ such that 
\begin{enumerate}
    \item $x\cdot(y\cdot z)=(x\cdot y)\cdot z$; 
    \item $x\cdot (x\backslash y)=y$;
    \item $x\backslash (x\cdot y)=y$,
\end{enumerate}
for all $x,y,z\in S$. Then $(S,\cdot)$ is a right group. Indeed (1) is equivalent to saying that $\cdot$ is associative, (2) implies that all the mappings $\ell_s$ ($s\in S$) are surjective and (3) that all the mappings $\ell_s$ are injective, since if $t,u\in S$ satisfy $s\cdot t=s\cdot u$, then
$$
t=s\backslash(s\cdot t)=s\backslash (s\cdot u)=u. 
$$

Hence, the variety of right groups can be defined as the variety of all algebras $(S,\cdot,\backslash)$, where $S$ is a set and $\cdot,\backslash$ are two binary operations on $S$ satisfying the identities $x\cdot(y\cdot z)=(x\cdot y)\cdot z$, \ $x\cdot (x\backslash y)=y$ and $x\backslash (x\cdot y)=y$. 

\begin{lemma}\label{1.2} {\rm \cite[Lemma~1.26]{CPI}} {\rm (a)} Every idempotent of a right simple semigroup $S$ is a left identity for $S$. 

{\rm (b)} Every idempotent of a left cancellative semigroup $S$ is a left identity for $S$.
\end{lemma}

\begin{proof} {\rm (a)} Assume $S$ right simple. If $e,x\in S$ and $e^2=e$, then $x=ey$ for some $y\in S$, so $ex=e(ey)=ey=x$.

{\rm (b)} Suppose $S$ left cancellative. If $e,x\in S$ and $e^2=e$, then $e(ex)=ex$ implies $ex=x$ because $S$ is left cancellative. 
\end{proof}

\begin{theorem}\label{1.1}  {\rm \cite[Section~1.11, Theorem~1.27]{CPI}} The following assertions are equivalent for a semigroup $S\ne\emptyset$:

{\rm (a)} $S$ is a right group.

{\rm (b)} $S$ is right simple and left cancellative.

{\rm (c)} For every $a,b\in S$ there exists a unique element $x\in S$ such that $ax=b$.

{\rm (d)} $S$ is right simple and contains an idempotent.

{\rm (e)} $S$ is isomorphic to the external direct product of a group $G$ and a non-empty right zero semigroup $E$.

{\rm (f)} There exists an element $e\in S$ such that: {\rm (1)}  $e$ is a left identity  for $S$, and {\rm (2)} every element of $S$ has a right inverse with respect to $e$.

{\rm (g)} {\rm (1)} $S$ has a left identity, and {\rm (2)} for every left identity $e$ of $S$ and every element $a\in S$, $a$ has a right inverse with respect to $e$.\end{theorem}

Notice that if $S$ a left cancellative semigroup that contains an idempotent, then we cannot deduce that $S$ is a right group, as the example of the additive monoid $\N_0$ shows. Cf.~Theorem \ref{1.1}(d).

\bigskip

Some care is necessary as far as 
the direct-product representation of a right group $(S,\cdot)$ as the direct product of a group $G$ and a non-empty right zero semigroup $E=\E(S,\cdot)$ in 
Theorem~\ref{1.1}(e) is concerned. Namely, for every $e\in E$, there is an idempotent endomorphism of $S$ defined by setting $\sigma_e(s) = s \cdot e$ for every $s \in S$. The image of $\sigma_e$ is the subgroup $S \cdot e$ of $S$, which is a group isomorphic to $G$.

Also, there is an idempotent endomorphism $\tau$ of the semigroup $S$ defined by $\tau(s)=s\backslash s$ for every $s\in S$. The image of $\tau$ is the subsemigroup $E$ of $S$.  The semigroup $S$ is the product of $G$ and $E$ in the category of semigroups, and, for every $e\in E$, the pointed semigroup $(S,\cdot, e)$ is the coproduct of $(S\cdot e,\cdot, e)$ and $(E,\pi_2,e)$ in the category of pointed right groups \cite[Proposition~3.1]{FacFin}. 
Thus $(S,\cdot, e)$ is the biproduct of $(S\cdot e,\cdot, e)$ and $(E,\pi_2,e)$ in the category of pointed right groups. 

For a right group $(S,\cdot)$, let $E$ denote the set of all idempotents of $S$, and
 fix an element $e_0\in E$. Define an equivalence relation $\sim $ on $S$ setting, for every $a,b\in S$, $a\sim b$ if $a\cdot e_0=b\cdot e_0$. The equivalence $\sim$ does not depend on the choice of the idempotent $e_0$, since, for every $a,b\in S$ and every $e\in E$, one has that $a\cdot e=b\cdot e$ if and only if $a\cdot e_0=b\cdot e_0$. If fact, if $a\cdot e=b\cdot e$, multiplying by $e_0$ on the right, one gets that $a\cdot e_0=b\cdot e_0$. Similarly, $a\cdot e_0=b\cdot e_0$ implies $a\cdot e=b\cdot e$.
 
The equivalence $\sim$ is compatible with both the operations $\cdot$ and $\backslash$ on $S$. 
The equivalence class of any $a\in S$ modulo $\sim$ is $a\cdot E:=\{\,a\cdot e\mid e\in E\,\}$, so that there is a partition $\{\,a\cdot E\mid a\in S\,\}$ of $S$. A complete irredundant set of representatives of the congruence classes of $S$ modulo $\sim$ is the set $S\cdot e_0$. The groups $S\cdot e_0$ and $S/{\sim}$ are canonically isomorphic.

\begin{proposition}\label{inv} {\rm \cite[Proposition 4.3]{FacFin}} For a right group $S$, there is a one-to-one correspondence between the set of the semigroup morphisms that are right inverses of the canonical projection $\pi\colon S\to S/{\sim}$ and the set $\E(S)$. If $e_0\in \E(S)$, the right inverse homomorphism of $\pi$ corresponding to $e_0$ is the semigroup morphism $\overline{r_{e_0}}\colon S/{\sim}\to S$ induced by the right multiplication $r_{e_0}\colon S\to S$ by $e_0$.\end{proposition}

The kernel of the idempotent endomorphism $\tau$ of the semigroup $S$ is a congruence $\equiv$ on $S$ and its image is $E$. We have that $\tau(x)=x\backslash x$ is the unique element $f\in E$ such that $xf=x$. Every right group $S$ has a partition $\{\,S\cdot e\mid e\in E\,\}$; this easily follows from the fact that, for each $e\in E$, $S\cdot e$ is a group with identity element $e$.  

\section{First results on disemigroups and digroups}

\begin{theorem}\label{2} Let $(D, \vdash,\dashv)$ be a disemigroup and $\U(D)$ the set of its bar-units. Let $\sim$ be the congruence on $(D, \vdash,\dashv)$ generated by the subset $$\{\,(x\vdash y,\ x\dashv y)\mid x,y\in D\,\}$$ of $D\times D$. Then:

{\rm (1)} $e\sim f$ for all $e,f\in \U(D)$. 

{\rm (2)} The functor $F\colon {\SGrp}\to\DSGp$ has a left-inverse left-adjoint $$H\colon \DSGp\to {\SGrp}$$ that associates to every disemigroup $D$ the semigroup $D/{\sim}$.\end{theorem}

\begin{proof} (1) Let $e$ and $f$ be two elements of $\U(D)$. Since $e\in \U(D)$, we have that $e\vdash f=f\dashv e=f$. Similarly, since $f\in \U(D)$, we know that $f\vdash e=e\dashv f=e$. Now in the congruence $\sim$ we have the pair $(f\vdash e, f\dashv e)=(e,f)$. Thus we get $e\sim f$, as desired.

(2) In order to prove that $H$ is a left-inverse of $F$, that is, that $H\circ F$ is naturally isomorphic to the identity functor $1_{\SGrp}$, notice that for every semigroup $(S,\cdot)$ we have that $(S,\cdot,\cdot)/{\sim}$ is canonically isomorphic to $(S,\cdot,\cdot)$ because the congruence $\sim$ is the equality $=$ on $S$.

Finally, in order to show that $H$ is a left-adjoint of $F$, fix any disemigroup $D$ and any disemigroup morphism $\varphi\colon D\to (S,\cdot,\cdot)$, where $(S,\cdot)$ is any semigroup. Now $\sim$ is the congruence on $(D, \vdash,\dashv)$ generated by $$\{\,(x\vdash y,\ x\dashv y)\mid x,y\in D\,\}.$$ For any pair $(x\vdash y, x\dashv y)$, with $x,y\in D$, we have that $\varphi(x\vdash y)=\varphi(x)\cdot\varphi(y)=\varphi(x\dashv y)$, hence the pair $(x\vdash y, x\dashv y)$ belongs to the kernel of the morphism $\varphi$. Therefore there is a unique disemigroup morphism $\widetilde{\varphi}\colon D/{\sim}\to (S,\cdot,\cdot)$ such that $\widetilde{\varphi}\circ\pi=\varphi$, where $\pi\colon D\to D/{\sim}$ is the canonical projection. This proves that $H$ is a left-adjoint for $F$.
\end{proof}

\begin{definition} {\rm A {\em digroup}  is a disemigroup  $(D,\vdash,\dashv)$ for which $\I(D,\vdash,\dashv)$ is non-empty. }\end{definition}

We will seen in Section~ref{7} that for a digroup $D$ the set $\I(D)$ is much more important than the halo $\U(D)$.

	 \begin{proposition}\label{3} {\rm (a)} If $D$ is a disemigroup and $\U(D)\ne\emptyset$, then the semigroup $D/\!\sim$ is a monoid.

{\rm (b)} If $D$ is a digroup, then the semigroup $D/\!\sim$ is a group. \end{proposition}

\begin{proof} (a) If $\U(D)\ne\emptyset$, fix an element $e\in \U(D)$. By Theorem~\ref{2}(1), the congruence class $[e]_\sim$ contains $\U(D)$ and is a two-sided identity of the semigroup $D/\!\sim$, because the two operations coincide on $D/\!\sim$.

(b) If $D$ is a digroup, its homomorphic image $D/\!\sim$ is also a digroup. Moreover, the two operations coincide on $D/\!\sim$, hence $D/\!\sim$ is a group.
\end{proof}

Similarly to Lemma \ref{2.4}, there is a full and faithful functor $F\colon \Grp\to \Dg$ of the category ${\Grp}$ of groups into the category $\Dg$ of digroups that associates to every group $(G,\cdot)$ the digroup $(G,\cdot,\cdot)$. Hence there is a canonical category isomorphism of the category ${\Grp}$ of groups into the full subcategory of $\Dg$ whose objects are the digroups $(D, \vdash,\dashv)$ for which the two operations $\vdash$ and $\dashv$ coincide, and therefore we will call {\em groups}  the digroups $(D, \vdash,\dashv)$ for which the two operations $\vdash$ and $\dashv$ coincide (or, equivalently, $\U(D, \vdash,\dashv)$ has exactly one element). For a group $(G,\cdot)$ one has that $\U(G,\cdot,\cdot)=\I(G,\cdot,\cdot)=\{1_G\}$.

\begin{remark}  {\rm In general, for a disemigroup $D$, one has that $\sim$ can contain $\{\,(x\vdash y, x\dashv y)\mid x,y\in D\,\}$ properly. For example, let $D$ be any set with at least two elements. Fix an element $g_0\in D$. Let $\cdot$ be the operation on $D$ defined by $x\cdot y=g_0$ for all $x,y\in D$. Then $\cdot$ is associative, $(D,\cdot)$ is a semigroup, and $(D,\cdot,\cdot)$ is a disemigroup. In this case the set $\{\,(x\vdash y, x\dashv y)\mid x,y\in D\,\}$ is the singleton $\{(g_0,g_0)\}$. This singleton generates the congruence $=$, which is equipotent to $D$. Therefore the equality congruence $\sim$ properly contains the set $\{\,(x\vdash y, x\dashv y)\mid x,y\in D\,\}$. }\end{remark}

It is easy to see that, for any non-empty set $X$, the triplet $(X,\pi_2,\pi_1)$ is a digroup. All its elements are bar-units. 
If $e,x\in X$, then $x^{\dagger_e}=e$ is a simultaneous inverse of $x$ with respect to $e$. Therefore $\U(X)=\I(X)=X$ in this digroup. 
Moreover $\sim$ is the trivial congruence, hence $X/\!\sim$ is the group with one element. There is functor $G\colon {\Set_{\ne\emptyset}}\to\Dg$ of the category $\Set_{\ne\emptyset}$ of non-empty sets into the category $\DSGp$ of digroups that associates to every non-empty set $X$ the digroup $(X,\pi_2,\pi_1)$. The functor $G$ is the identity on morphisms, and is full and faithful, so that it induces a canonical category isomorphism of the category $\Set_{\ne\emptyset}$ of non-empty sets into the full subcategory of $\Dg$ whose objects are the digroups $(D, \vdash,\dashv)$ for which $\vdash$ coincides with $\pi_2$ and $\dashv$ coincides with $\pi_1$.

\begin{lemma} Let $(S,\cdot)$ be a right group. Then there exists a binary operation $\dashv$ on the set $S$ such that $(S,\cdot,\dashv)$ is a digroup.\end{lemma}

\begin{proof} Since $(S,\cdot)$ is a right group, $(S,\cdot)$ is isomorphic to the external direct product of a group $G$ and a non-empty right zero semigroup $E$ . For simplicity of notation, we will suppose $S=G\times E$. Define the binary operation $\dashv$ on $S$ setting $(g,e)\dashv(h,f)=(g\cdot h, e)$ for every $(g,e),(h,f)\in G\times E$. We leave to the reader the simple verification that $(S,\cdot,\dashv)$ is a digroup. One has $\U(S,\cdot,\dashv)=\I(S,\cdot,\dashv)=\{1_G\}\times E$.\end{proof}

	 \begin{proposition}\label{3.9} Let $(D,\vdash,\dashv)$ be a digroup. Then:
	 
	\noindent{\rm (a)} The semigroup $(D,\vdash)$ is a right group.
	 
\noindent{\rm	 (b)} The semigroup $(D,\dashv)$ is a left group.
	 
	\noindent{\rm (c) }The following four conditions are equivalent for an element $x\in D$:
	 
	{\rm (1) }$x\vdash x=x$;
	 
{\rm (2)}  $x\dashv x=x$;

{\rm	(3)} $x\vdash y=y$ for every $y\in D$;
	
{\rm (4)}  $y\dashv x=y$ for every $y\in D$.
	
\noindent{\rm 	(d)} An element $e\in \U(D)$ belongs to $\I(D)$ if and only if $x\vdash e=e\dashv x$ for every $x\in D$.

\noindent{\rm 	(e)} $\emptyset\ne\I(D,\vdash,\dashv)\subseteq\U(D,\vdash,\dashv)=\E(D,\vdash)=\E(D,\dashv)$.\end{proposition}
	
	\begin{proof} If $(D,\vdash,\dashv)$ is a digroup, the set $\I(D,\vdash,\dashv)$ is non-empty, so that we can fix an element $e$ in $\I(D,\vdash,\dashv)$. Then, for all $x\in D$, we have that $e\vdash x=x$, $x\dashv e=x$, and there exists $x^{\dagger_e}\in D$ such that $x\vdash x^{\dagger_e}=e$ and $x^{\dagger_e}\dashv x=e$.
	
	(a) Fix $a,b\in D$. Then $a\vdash(a^{\dagger_e}\vdash b)=e\vdash b=b$. This proves that the semigroup $(D,\vdash)$ is right simple. But $\emptyset\ne\I(D,\vdash,\dashv)\subseteq\E(D,\vdash)$, so that $(D,\vdash)$ is a right group by Theorem~\ref{1.1}(d).
	
	(b)	This is the right/left symmetric of (a) (Remark~\ref{simmetria}).
	
	(c) The implication (3)${}\Rightarrow{}$(1) is trivial. The implication (1)${}\Rightarrow{}$(3) holds by (a) and Lemma~\ref{1.2}(a). Therefore we get that (1)${}\Leftrightarrow{}$(3) holds. The equivalence (1)${}\Leftrightarrow{}$(3) is its right/left symmetric (Remark~\ref{simmetria}).
	
	The proof of (3)${}\Rightarrow{}$(4) is given in \cite[Lemma~4.4]{Kinyon}, and (4)${}\Rightarrow{}$(3) is  its right/left symmetric.
	
	(d) If $e\in\I(D)$, then every element of $D$ has a simultaneous inverse with respect to $e$. Thus, for any $x\in D$, there is an element $x^{\dagger_e}\in D$ such that $x\vdash x^{\dagger_e}=x^{\dagger_e} \dashv x=e$. It follows that
    $$
    x\vdash e=x\vdash (x^{\dagger_e} \dashv x)=(x\vdash x^{\dagger_e})\dashv x=e\dashv x.
    $$

	
	For the converse, fix an element $f\in\I(D)$, which exists because $D$ is a digroup. For an element $e\in \U(D)\setminus\I(D)$, there exists an element $x\in D$ for which the two inverses $x_r$ and $x_l$ of $x$ relative to $e$ are different, that is, there exist elements 
	$x,x^{\dagger_f}, x_r,x_l\in D$ such that $x\vdash x^{\dagger_f}=x^{\dagger_f} \dashv x=f$, $x\vdash x_r=e$, $x_l\dashv x=e$ and $x_r\ne x_l$. Let us show that $x_r=x^{\dagger_f} \vdash e$. By the uniquenes of the right inverse $x_r$ in the right group $(D,\vdash)$, it suffices to show that $x\vdash (x^{\dagger_f} \vdash e)=e$. This is trivially true. By left/right symmetry, $x_l=e\dashv x^{\dagger_f} $. Therefore $x_r\ne x_l$ can be rewritten as $x^{\dagger_f} \vdash e\ne e\dashv x^{\dagger_f} $. Hence the condition $y\vdash e=e\dashv y$ for every $y\in D$ does not hold.
	
	(e) follows immediately from (c).
	\end{proof}

By Proposition~\ref{3.9}(c) it would be better to call {\em idempotents} the elements of $\U(D)$ for a digroup $D$, rather than bar-units, and, by (d) it would be better to call {\em central idempotents} the elements of $\I(D)$

    \medskip
    
	From Proposition~\ref{3.9}((a) and (b)), we know that in any digroup $(D,\vdash,\dashv)$, the right inverse relatively to $\vdash$ and the left inverse relatively to $\dashv$, of any element $x\in D$, with respect to any $e\in U(D)$ are both unique. We will show in Example ~\ref{5.2} that there exist digroups $D$ with $\I(D)\ne \U(D)$. By definition of $\I(D)$, for any digroup $(D,\vdash,\dashv)$ and any element $e\in\U(D)\setminus\I(D)$, there exists an element $x\in D$ for which the unique right inverse of $x$ relatively to $\vdash$ is different from the unique left inverse of $x$ relatively to $\dashv$.
	
	\medskip
	
	The disemigroups $D$ with $\U(D)\ne\emptyset$ are called {\em dimonoids} in \cite{Kinyon}. The disemigroups $(D,\vdash,\dashv)$ with $\U(D)\ne\emptyset$,  $(D,\vdash)$  a right group and
$(D,\dashv)$ a left group. are called {\em generalized digroups} in \cite{Salazar}. 
	
	\medskip

J.~D.~Phillips \cite{Phillips} has shown that, for a set $D$ with two binary operations $\vdash$ and $\dashv$, one has that $(D,\vdash,\dashv)$ is a digroup if and only if

(1) $(D,\vdash)$ and $(D,\dashv)$ are semigroups;
	
    (2) $x\vdash (y\dashv z)=(x\vdash y)\dashv z$ for every $x,y,z\in D$; and
	
    (3) $\I(D,\vdash,\dashv)\neq\emptyset$

\medskip

 A {\em pointed digroup} is a pair $(D,e_0)$, where $D$ is a digroup and $e_0\in I(D)$. A {\em pointed right group} is a pair $(S,e_0)$, where $S$ is a right group and $e_0\in \E(S)$. In the category of pointed digroups, morphisms $$f\colon (D,\vdash,\dashv,e_0)\to(D',\vdash,\dashv,e'_0)$$ are the mappings $f\colon D\to D'$ that preserve the binary operations $\vdash$ and $\dashv$ and map the {\em base point} $e_0$ to the base point $e'_0$. Similarly for pointed right groups $(S,e_0)$, where we suppose that the base point $e_0$ belongs to $\E(S)$.

\section{Idempotent digroup endomorphisms}

As we have already mentioned after the statement of Theorem~\ref{1.1}, if $(S,\cdot)$ is a right group and $e\in \E(S,\cdot)$, there is an idempotent endomorphism of $(S,\cdot)$ defined by setting $\sigma_e(s) = s \cdot e$ for every $s \in S$. The image of $\sigma_e$ is the subsemigroup $S \cdot e$ of $S$, and  $S \cdot e$ is a group. Let $(D,\vdash,\dashv)$ be a digroup. Then, for every $e\in \U(D,\vdash,\dashv)=\E(D,\vdash)=\E(D,\dashv)$, there is an idempotent semigroup endomorphism of $(D,\vdash)$ defined by setting $\sigma_e(s) = s \vdash e$ for every $s \in D$ and there is idempotent endomorphism of $(D,\dashv)$ defined by setting $\sigma'_e(s) = e \dashv s$ for every $s \in D$. If $e\in \I(D)$, then $s\vdash e=e\dashv s$ for every $s\in D$, so that $\sigma_e=\sigma'_e$ is an idempotent  digroup morphism. More precisely we have:

\begin{proposition} Let $(D, \vdash,\dashv)$ be a digroup and let $e$ be an element of $\U(D)$. Then:

\noindent{\rm (a)} The mapping $r_{\vdash e}\colon D\to D$, defined by $r_{\vdash e}(x)=x\vdash e$ for every $x\in D$, is a digroup morphism if and only if $e\in \I(D)$.

\noindent{\rm (b)} There is a one-to-one correspondence between the set of the digroup morphisms that are right inverses of the canonical projection $\pi\colon D\to D/{\sim}$ and the set $\I(D)$. If $e\in \I(D)$, the inverse homomorphism of $\pi$ corresponding to $e$ is the semigroup morphism $\overline{r_{\vdash e}}\colon D/{\sim}\to D$ induced by right multiplication $r_{\vdash e}\colon D\to D$ by $e$.\end{proposition}

\begin{proof} 
(a) Let $e$ be an element of $\U(D)$. Since $(D, \vdash)$ is a right group,  we know that the mapping $r_{\vdash e}\colon D\to D$, defined by $r_{\vdash e}(x)=x\vdash e$ for every $x\in D$, is a  right group morphism (Proposition~\ref{inv} or \cite[Proposition 4.3]{FacFin}). 
Therefore $r_{\vdash e}$ is a digroup morphism if and only if a left group endomorphism of the left group $(D,\dashv)$, that is, if and only if $(x\vdash e)\dashv(y\vdash e)=(x\dashv y)\vdash e$ for every $x,y\in D$.     

Now if $e\in \I(D)$, then $t\vdash e=e\dashv t$ for every $t\in G$ (Proposition~\ref{3.9}(d)). Therefore $(x\vdash e)\dashv(y\vdash e)=(e\dashv x)\dashv(e\dashv y)=e\dashv x\dashv y=(x\dashv y)\vdash e$, as desired.

Conversely, if $(x\vdash e)\dashv(y\vdash e)=(x\dashv y)\vdash e$ for every $x,y\in D$ and we replace $x$ with $e$, we get that, for every $y\in D$, $(e\vdash e)\dashv(y\vdash e)=(e\dashv y)\vdash e$, so $e\dashv(y\vdash e)=(e\vdash y )\vdash e$, i.e., $e\dashv(y\dashv e)=y\vdash e$. That is,  
$e\dashv y=y\vdash e$, as we wanted to prove.

(b) now follows from Proposition~\ref{inv}.\end{proof}

In the notation of the previous proposition, we have that the image of $\sigma_e$ is $S \vdash e= e \dashv S$ for every $e\in \I(D)$. As far as its kernel is concerned we have (cf.~Theorem \ref{2}):

\begin{proposition} Let $(D, \vdash,\dashv)$ be a digroup. For every $e\in \I(D)$, the kernel of the idempotent digroup endomorphism $\sigma_e$, that is, the congruence $\sim$ on $D$ defined, for every $x,y\in D$, by $x\sim y$ if $x \vdash e=y\vdash e$ (equivalently, if $e\dashv x=e\dashv y$) is the congruence  on $(D, \vdash,\dashv)$ generated by the subset $$\{\,(x\vdash y,\ x\dashv y)\mid x,y\in D\,\}$$ of $D\times D$. \end{proposition}

\begin{proof} Given any $x,y\in D$, one has that $x\vdash y\sim x\dashv y$, because $(x\vdash y)\vdash e=(x\dashv y)\vdash e$.

Conversely, suppose $x,y\in D$ and $x\sim y$, so that $x \vdash e=y\vdash e$. Let $\sim'$ denote the congruence  on $(D, \vdash,\dashv)$ generated by the set $\{\,(x\vdash y,\ x\dashv y)\mid x,y\in D\,\}$, and let $\backslash$ denote the left inverse operation of $\vdash$, so that $x=x \vdash e\vdash (x\backslash x)$, $y=y \vdash e\vdash (y\backslash y)$, and $x\backslash x,y\backslash y \in\E(D, \vdash)$. Then $(x\backslash x)\sim'(y\backslash y)$ by Theorem \ref{2}(1). Since $\sim'$ is a congruence, it follows that  $x=x \vdash e\vdash (x\backslash x)\sim' y \vdash e\vdash (y\backslash y)=y$, as desired.\end{proof} 

On the contrary, the two semigroup endomorphisms $\tau_\vdash$ and $\tau_\dashv$ of $(D, \vdash$ and $(D,\dashv)$ respectively, do not coincide in general. In fact, let $D$ be a digroup and $e$ an element in $U(D)\setminus I(D)$. Then there exists $y\in D$ such that $e\dashv y\ne y\vdash e$. Set $x:=y\vdash e$. Let us show that $\tau_\vdash(x)\ne\tau_\dashv(x)$. We have that $\tau_\vdash(x)=e$, because $x\vdash e=y\vdash e\vdash e=y\vdash e=x$. Assume by contradiction that $\tau_\dashv(x)=e$. Then $e\dashv x=x$, that is, $e\dashv (y\vdash e)=y\vdash e$, i.e., $e\dashv y\dashv e=y\vdash e$. Equivalently, $e\dashv y=y\vdash e$. This is a contradiction.

\bigskip

Digroups are obtained gluing together a left group and a right group. Let us describe how it is possible to construct such a gluing. First of all, notice that in a digroup $(D,\vdash,\dashv)$ there is semigroup homomorphism that describes how far the right group structure $(D,\vdash)$ and the left group structure $(D,\dashv)$ are. For every $x,y\in D$ there exists a unique element $\lambda_x(y)$ such that $x\vdash y=\lambda_x(y)\dashv x$ because $(D, \dashv)$ is a left group. Let us examine the properties of this mapping $\lambda \colon D\to D^D$. 

\begin{proposition} $\lambda$ is a digroup homomorphism of the digroup  $(D,\vdash,\dashv)$ into the group $\Aut_\Dg(D,\vdash,\dashv)$.\end{proposition}

\begin{proof} Let us prove that $\lambda_x\colon D\to D$ is an endomorphism of $(D,\vdash,\dashv)$, that is, that $\lambda_x(y\vdash z)=\lambda_x(y)\vdash\lambda_x(z)$ and $\lambda_x(y\dashv z)=\lambda_x(y)\dashv\lambda_x(z)$ for every $x,y,z\in D$. 

We have that $\lambda_x(y\vdash z)\dashv x=x\vdash y\vdash z=(\lambda_x(y)\dashv x)\vdash z=\lambda_x(y)\vdash x\vdash z=\lambda_x(y)\vdash(\lambda_x(z)\dashv x)=(\lambda_x(y)\vdash\lambda_x(z))\dashv x$. Since $(D,\dashv)$ is right cancellable, it follows that 
$\lambda_x(y\vdash z)=\lambda_x(y)\vdash\lambda_x(z)$.

Similarly, $\lambda_x(y\dashv z)\dashv x=x\vdash y\dashv z=\lambda_x(y)\dashv x\dashv z=\lambda_x(y)\dashv( x\vdash z)=\lambda_x(y)\dashv(\lambda_x(z)\dashv x)=(\lambda_x(y)\dashv\lambda_x(z))\dashv x$. Since $(D,\dashv)$ is right cancellable, it follows that 
$\lambda_x(y\dashv z)=\lambda_x(y)\dashv\lambda_x(z)$. This proves that every $\lambda_x\colon D\to D$ is a digroup endomorphism.

Each $\lambda_x$ is a bijection, because for every $z\in D$ there exists a unique element $y\in D$ such that $x\vdash y=z\dashv x$. Equivalently, for every $z\in D$ there exists a unique $y\in D$ such that $z=\lambda_x(y)$. Therefore each $\lambda_x$ is a bijection, so that $\lambda$ can be seen as a mapping of $D$ into the group $\Aut_\Dg(D,\vdash,\dashv)$.

Let us prove that $\lambda$ is a semigroup morphism of the semigroup $(D,\vdash)$ into the group $\Aut_\Dg(D,\vdash,\dashv)$, i.e., that $\lambda_{x\vdash y}=\lambda_x\circ\lambda_y$ for every $x,y\in D$. Apply $\lambda_x$ to the equality $y\vdash z=\lambda_y(z)\dashv y$, getting $\lambda_x(y)\vdash \lambda_x(z)=(\lambda_x\circ\lambda_y)(z)\dashv \lambda_x(y)$. It follows that $(\lambda_x(y)\vdash \lambda_x(z))\dashv x=((\lambda_x\circ\lambda_y)(z)\dashv \lambda_x(y))\dashv x$, which can be rewritten as $\lambda_x(y)\vdash x \vdash z=(\lambda_x\circ\lambda_y)(z)\dashv( x\vdash y)$. Then $\lambda_{x\vdash y}(z)\dashv(x\vdash y)=x\vdash y\vdash z=(\lambda_x(y)\dashv x)\vdash z=\lambda_x(y)\vdash x\vdash z=(\lambda_x\circ\lambda_y)(z)\dashv( x\vdash y)$. Since $(D,\dashv)$ is right cancellable, it follows that $\lambda_{x\vdash y}(z)=(\lambda_x\circ\lambda_y)(z)$. Hence $\lambda$ is a semigroup morphism of $(D,\vdash)$ into the group $\Aut_\Dg(D,\vdash,\dashv)$. 

In particular, $\lambda$ maps all elements of $\E(D,\vdash)$ to the identity $\id_D$. It follows that the congruence $\sim$ is contained in the kernel of the semigroup morphism $\lambda$, because if $x,y\in D$ and $x\sim y$, for an idempotent element $e\in E(D,\vdash)$ we have that $x\vdash e=y\vdash e$, so that $\lambda_{x\vdash e}=\lambda_{y\vdash e}$, hence $\lambda_x=\lambda_y$, that is, $x$ and $y$ are congruent modulo the kernel of $\lambda$. Thus $\lambda$ induces a group morphism $\overline{\lambda}\colon D/{\sim}\to \Aut_\Dg(D,\vdash,\dashv)$. Then $\lambda$ is also a digroup morphism, because, for every $x,y\in D$, $\lambda(x\dashv y)=\overline {\lambda}([x\dashv y]_\sim)=\overline {\lambda}([x\vdash y]_\sim)={\lambda}(x\vdash y)=\lambda(x)\circ\lambda(y)$.
\end{proof}

This semigroup homomorphism $\lambda$ describes how far the right group structure $(D,\vdash)$ is from the left group structure $(D,\dashv)$. It is very similar to the mapping $\lambda$ of left skew braces or digroups $(D,+, \circ)$, where digroup here is in the sense of \cite{7}. 

\section{Constructions of digroups}

In the previous section we have said that digroups are a gluing of a left group and a right group. Let us see how it is possible to construct a digroup starting from a a group $G$ and a set $X$ that is at the same time a left $G$-set and a right $G$-set (i.e., there are a group homomorphism $G\to\Sym_X$ and a group antihomomorphism $G\to\Sym_X$.)

Let $G$ be a group and $X$ be a left $G$-set, so that $X\times G$ is a  right group with respect to the operation $\vdash$ defined by $(x,g)\vdash (y,h)=(gy, gh)$ for every $x,y\in X$, $g,h\in G$ \cite[Theorem~3.3]{FacFin}. Also, suppose that $X$ also has a right $G$-set structure $X\times G\to X$, $(x,g)\mapsto xg$, so that $X\times G$ is also a left group with respect to the operation $\dashv$ defined by $(x,g)\dashv (y,h)=(xh, gh)$ for every $x,y\in X$, $g,h\in G$. Compatibility of the two operations $\vdash,\dashv$ on $X\times G$ is equivalent to $(gx)h=g(xh)$ for all $x\in X$ and all $g,h\in G$. The elements of $X\times G$ idempotent with respect to both operations $\vdash$ and $\dashv$ are exactly the elements of the form $(x,1_G)$ for every $x\in X$. It's easy to compute that for every element $(x,g)\in X\times G$, the right inverse of $(x,g)$ with respect to the operation $\vdash$ and a right identity $(x_0,1_G)$ is $(g^{-1}x_0, g^{-1})$, and the left inverse of $(x,g)$ with respect to the operation $\dashv$ and the left identity $(x_0,1_G)$ is $(x_0g^{-1}, g^{-1})$. Therefore the two inverses coincide if and only if $x_0g^{-1}=g^{-1} x_0$ for every $g\in G$, that is, if and only if $x_0g=gx_0$ for every $g\in G$.

Fix an element $x_0\in X$ with  $x_0g=gx_0$ for every $g\in G$, so that the two operations $\vdash$ and $\dashv$ on the subset $\{x_0\}\times G$ of $X\times G$ coincide, because $(x_0,g)\vdash(x_0,h)=(gx_0,gh)$ and $(x_0,g)\dashv(x_0,h)=(x_0g,gh)$. The element $\lambda_{(x,g)}(y,h)$ is $(gyg^{-1},ghg^{-1}).$ Notice that $\lambda$ does not depend on the choice of $x_0$, and that $(x_0,1_G)$ is fixed by all these mappings $\lambda_{(x,g)}$. Also, the mapping $\ell\colon G\to\Aut_\Set(X)$, $\ell\colon g\to \ell_g$, with $\ell_g(x)=gxg^{-1}$, is a left action of $G$ on $X$, and the mapping $r\colon G\to\Aut_\Set(X)$, $r\colon g\mapsto r_g$, with $r_g(x)=g^{-1}xg$, is a right action of $G$ on $X$. 
But:

(1) consider the digroup $(X\times G,\vdash,\dashv)$, constructed from the two compatible actions $G\times X\to X$, $(g,x)\mapsto gx$, (the left action of $G$ on $X$), and $X\times G\to X$, $(x,g)\mapsto xg$, (the right left action of $G$ on $X$), in which the operations are defined by $(x,g)\vdash (y,h)=(gy, gh)$ and $(x,g)\dashv (y,h)=(xh, gh)$  for every $x,y\in X$, $g,h\in G$; and

(2) consider the digroup $(X\times G,\vdash',\dashv')$, constructed from the two compatible actions $\pi_2\colon G\times X\to X$, $(g,x)\mapsto x$, and $X\times G\to X$, $(x,g)\mapsto g^{-1}xg$, in which the operations are defined by $(x,g)\vdash' (y,h)=(y, gh)$ and $(x,g)\dashv'(y,h)=(h^{-1}xh, gh)$  for every $x,y\in X$, $g,h\in G$. 

Then the digroups $(X\times G,\vdash,\dashv)$ and $(X\times G,\vdash',\dashv')$ are isomorphic. The isomorphism $(X\times G,\vdash,\dashv)\to (X\times G,\vdash',\dashv')$ is defined by $(x,g)\mapsto(xg^{-1},g)$.  The two actions on the digroup $(X\times G,\vdash',\dashv')$ are constructed from the trivial left $G$-action on $X$ that maps all the elements of $G$ to the identity mapping $X\to X$, and the right action given by $r\colon G\to\Aut_\Set(X)$, $r\colon g\mapsto r_g$, with $r_g(x)=g^{-1}xg$ for every $g\in G$ and $x\in X$.

Similarly, and dually, the digroup $(X\times G,\vdash,\dashv)$ is isomorphic to the digroup $(X\times G,\vdash'',\dashv'')$, constructed from the two compatible actions $\lambda\colon G\times X\to X$, $(g,x)\mapsto gxg^{-1}$, and $\pi_1\colon X\times G\to X$, $(x,g)\mapsto x$, in which the operations are defined by $(x,g)\vdash''(y,h)=(gyx^{-1}, gh)$ and $(x,g)\dashv'' (y,h)=(x, gh)$ for every $x,y\in X$, $g,h\in G$. The two actions on the digroup $(X\times G,\vdash'',\dashv'')$ are constructed from the left action given by $\lambda\colon G\to\Aut_\Set(X)$, with $\lambda_g(x)=gxg^{-1}$ for every $g\in G$ and $x\in X$, and the trivial right $G$-action on $X$ that maps all the elements of $G$ to the identity mapping $X\to X$.

\begin{Ex}\label{5.2} {\rm Let us present \cite[Example 4.2]{Kinyon} in our notations. 
Let $H$ be a group, $M\neq\emptyset$ be a set and let $$\diamond:H\times M\to M, \qquad (h,m)\mapsto h\diamond m,$$ be a left action of $H$ on $M$. Recall that an element $m\in M$ is said to be a fixed point of $M$, with respect to $\diamond$, if $h\diamond m=m$ for all $h\in H$.  Consider the operations $\vdash,\dashv$ on $D:=M\times H$ defined by 
$$
(u,h)\vdash(v,k):=(h\diamond v,hk),
$$
$$
(u,h)\dashv(v,k):=(u, hk),
$$
for all $(u,h),(v,k)\in D$. The following properties are straightforward. 

(1) $(D,\vdash,\dashv)$ is a disemigroup. 

    (2) The set of all bar-units of $D$ is 
	$$
	\U(D)=\{(f,1)\mid f\in M\},
	$$
	where $1$ is the identity element of $H$.
	
    (3) For every $f\in M$, consider the bar-unit $(f,1)$ of $D$, and let $(u,h)\in D$ be arbitrary. Then $(h^{-1}\diamond f,h^{-1})$ (resp., $(f,h^{-1})$) is the unique right (resp., left) inverse of $(u,h)$ in the semigroup $(D,\vdash)$ (resp., $(D,\dashv)$), with respect to $(f,1)$. Thus, $(u,h)$ admits a simultaneous inverse with respect to the bar-unit $(f,1)$ if and only if $h\diamond f=f$. It immediately follows that
	$$
	\I(D)=\{(f,1)\mid f \mbox{ is a fixed point of }M,\mbox{ with respect to }\diamond\}. 
	$$
	
    (4) By (3), $(D,\vdash,\dashv)$ is a digroup if and only if $M$ has a fixed point with respect to $\diamond$.  
\medskip

By the previous remarks, if the set of fixed points of $M$, with respect to $\diamond$, is a non-empty proper subset of $M$ then $(D,\vdash,\dashv)$ is a digroup such that $\I(D)\subsetneq \U(D)$. }
 \end{Ex}
 

%

 A \emph{digroup morphism} is a morphism $f\colon(D,\vdash,\dashv)\to (G,\vdash,\dashv)$ of disemigroups in the case in which both $D$ and $G$ are digroups. Let $(D,\vdash,\dashv)$ be a digroup. A subset $D'$ of $D$ is a \emph{subdigroup of $D$} if $D'$ is a digroup with respect to the restriction of the operations $\vdash,\dashv$ to $D'\times D'$.
If $D'$ is a subdigroup of $D$, then the inclusion is a morphism of digroups. We will freely use the following remarks where we collect some basic facts regarding digroups. 

\begin{remarks}\label{digroup-basic} {\rm
Let $(D,\vdash,\dashv)$ be a digroup. 

(1) If $D'$ is a subdigroup of $D$, it is possible to have that $\I(D')\neq \I(D)$ and $\U(D')\neq \U(D)$. For instance, suppose that $D$ is a digroup such that $\I(D)\subsetneq \U(D)$, and let  $e\in \U(D)\setminus \I(D)$. Then $D':=\{e\}$ is clearly a subdigroup of $D$ and $\U(D')=\I(D')=\{e\}$.

    (2) Let $f\colon(D,\vdash,\dashv)\to (G,\vdash,\dashv)$ be a morphism of digroups. Then $$f(\U(D,\vdash,\dashv))\subseteq \U(G,\vdash,\dashv),$$
	because $f(\E(D,\vdash))\subseteq \E(G,\vdash)$. 
	
    (3) If $f\colon(D,\vdash,\dashv)\to (G,\vdash,\dashv)$ is a constant morphism of digroups, then the unique element of $f(D)$ is a bar-unit of $G$, in view of (2). Conversely, given any bar-unit $g\in G$, the constant mapping $c_g\colon(D,\vdash,\dashv)\to (G\vdash,\dashv)$, $x\to g$, is a morphism of digroups.} 
\end{remarks}

\section{A pretorsion theory in the category of digroups}\label{3'}

We now recall the notions developed in \cite{FF} and \cite{fa-fi-gr} about pretorsion theories in arbitrary categories.
Let $\mathsf{C}$ be a category and $\mathsf{Z}$ be a non-empty class of objects of $\mathsf{C}$. For every pair $A,A'$ of objects of $\mathsf{C}$, we indicate by $\Triv_{\mathsf{Z}}(A, B)$ the set of  all morphisms in $\mathsf{C}$ that factor through an object of $\mathsf{Z}$. These morphisms are called {\em $\mathsf{Z}$-trivial}, or simply {\em trivial}.

If $f\colon A\to A'$ is a morphism in $\mathsf{C}$, a morphism $\varepsilon\colon X\to A$ in $\mathsf{C} $ is a \emph{$\mathsf{Z}$-prekernel} of $f$ if: 
\begin{enumerate}
	\item $f\varepsilon$ is a $\mathsf{Z}$-trivial morphism.
	\item If $\lambda \colon Y\to A$ is any morphism in $\mathsf{C}$ for which $f\lambda$ is $\mathsf{Z}$-trivial, then there exists a unique morphism $\lambda'\colon Y\to X$ in $\mathsf{C}$ such that $\lambda=\varepsilon\lambda'$. 
\end{enumerate}
Dually, a \emph{$\mathsf{Z}$-precokernel} of $f$ is a morphism $\eta\colon A'\to X$ such that:
\begin{enumerate}
	\item $\eta f$ is a $\mathsf{Z}$-trivial morphism.
	\item If $\mu\colon A'\to Y$  is any morphism in $\mathsf{C}$ for which $\mu f$ is $\mathsf{Z}$-trivial, then there exists a unique morphism $\mu'\colon X\to Y$ with $\mu=\mu' \eta$.
\end{enumerate}

If $f\colon A\to B$ and $g\colon B\to C$ are morphisms in $\mathsf{C}$, we say that $$\xymatrix{
	A \ar[r]^f &  B \ar[r]^g &  C}$$ is a \emph{short $\mathsf{Z}$-preexact sequence} in $\mathsf{C}$ if $f$ is a $\mathsf{Z}$-prekernel of $g$ and $g$ is a $\mathsf{Z}$-precokernel of $f$.

  \begin{definition}\label{pretorsion-theory-def} {\rm Let $\mathsf{C}$ be a category, and $\mathsf{T},\mathsf{F}$ be two replete (that is, closed under isomorphism) full subcategories of $\mathsf{C}$. Set $\mathsf{Z}:=\mathsf{T}\cap\mathsf{F}$. The pair $(\mathsf{T},\mathsf{F})$ is a {\em pretorsion theory}  in the category $\mathsf{C}$ if $\mathsf{Z}\ne\emptyset$ and the following properties hold.
  	
	(1) $\Hom_{\mathsf{C}}(T,F)=\Triv_{\mathsf{Z}}(T, F)$  for every object $T\in\mathsf{T}$, $F\in\mathsf{F}$.
	
	  (2) For every object $B$ of $\mathsf{C}$ there is a short $\mathsf{Z}$-preexact sequence $$\xymatrix{
	A \ar[r]^f &  B \ar[r]^g &  C}$$ with $A\in\mathsf{T}$ and $C\in\mathsf{F}$.}\end{definition}

In our case, the category $\mathsf{C}$ is the category $\Dg$ of digroups, the category $\mathsf{F}$ is the category of groups, that is, the category of the digroups $(D,\vdash,\dashv)$ with $\U(D,\vdash,\dashv)$ with exactly one element, and the category $\mathsf{T}$ is the category of all digroups $(X,\pi_2,\pi_1)$ with $X$ a non-empty set. Hence $\mathsf{Z}$ is the subcategory of $\Dg$ consisting of all digroups with exactly one element. Therefore, for every pair $A,A'$ of digroups, $\Triv_{\mathsf{Z}}(A, B)$ is the set of  all constant mappings $c_e\colon A\to B$, where $e\in\U(B)$ and $c_e(a)=e$ for every $a\in A$. Hence there is a bijection between $\Triv_{\mathsf{Z}}(A, B)$ and $\U(B)$.

The category $\Dg$ has no initial object. Indeed, let $D$ be a digroup with at least two distinct bar-units $e_1\neq e_2$ and let $X$ be any digroup; thus the constant mappings $f_i:X\to D$, $x\mapsto e_i$ ($i\in\{1,2\}$) are distinct morphisms in $\Dg$. It follows that in the category $\Dg$ there is no initial object. 

\begin{lemma}\label{Z-prekernel-charact}
Let $f:(D,\vdash,\dashv)\to (G,\vdash,\dashv)$ be a morphism of digroups and let $U:=\U(D,\vdash,\dashv)$. Then $f$ has a $\mathsf{Z}$-prekernel in $\Dg$ if and only if $f(U)$ is a singleton. 
Moreover, if the previous equivalent conditions are satisfied and $f(U)=\{g_0\}$, then $K:=f^{-1}(g_0)$ is a subdigroup of $D$ and the inclusion morphism $i:K\to D$ is a $\mathsf{Z}$-prekernel of $f$. 
\end{lemma}
\begin{proof}
First, assume that $f(U)=\{g_0\}$; in particular, $g_0$ is a bar-unit of $G$, by Remark \ref{digroup-basic}(2), and thus $g_0$ is idempotent with respect to $\vdash$ and $\dashv$. From this fact it immediately follows that $x\vdash y,x\dashv y\in K$ for every $x,y\in K$. Fix an element $e\in U$. Then, for every $x\in K$, we have
$$
g_0\vdash g_0=g_0=f(e)=f(x)\vdash f(x^{\dagger_e})=g_0\vdash f(x^{\dagger_e}), 
$$
and thus $f(x^{\dagger_e})=g_0$ because the semigroup $(G,\vdash)$ is left cancellative. It follows that $x^{\dagger_e}\in K$ for every $x\in K$, and thus $K$ is a subdigroup of $D$. By construction, the composite mapping $fi$ is trivial. Consider now any morphism $\lambda:Y\to D$ of digroups such that $f\lambda$ is trivial, that is, there is a bar-unit $g_1\in G$ such that $f(\lambda(Y))=\{g_1\}$. Since $Y$ has a bar-unit $y_0$ and $\lambda(y_0)\in U$, it immediately follows that $g_1=g_0$, proving that $\lambda(Y)\subseteq K$. Hence the mapping $\lambda':Y\to K$, $y\mapsto \lambda(y)$, is the unique morphism in $\Dg$ satisfying $\lambda=i\lambda'$. This proves that $i$ is a $\mathsf{Z}$-prekernel of $f$. 

Conversely, assume that $f$ has a $\mathsf{Z}$-prekernel $j:\Lambda\to D$. In particular, $fj$ is trivial, that is, $f(j(\Lambda))=\{g_0\}$ for some bar-unit $g_0\in G$. Since $\Lambda$ has some bar-unit $l_0$ and $j(l_0)\in U$, it follows that $g_0\in f(U)$. Consider now any bar-unit $e\in U$. Then the inclusion $\iota:\{e\}\to D$ is a morphism in $\Dg$ and the composite mapping $f\iota$ is a constant morphism of digroups, and thus it is trivial by Remark \ref{digroup-basic}(3). Since $j$ is a $\mathsf{Z}$-prekernel of $f$, there is a unique morphism $\iota_0:\{e\}\to \Lambda$ such that $\iota=j\iota_0$. It follows that $e\in j(\Lambda)$ and thus $f(j(\Lambda))=\{g_0\}$ implies that $f(e)=g_0$. This proves that $f(U)=\{g_0\}$, and the conclusion follows. 
\end{proof}
\begin{lemma}\label{pretorsion-ruciale}
Let $(D,\vdash,\dashv)$ be a digroup and let $\mu\colon D\to T$ be a digroup morphism such that $\mu(\U(D))$ is a singleton. Then the kernel $\sim$ of the canonical projection $\pi\colon D\to D/{\sim}$ is contained in the kernel of $\mu$. In particular, there exists a unique digroup morphism $\mu'\colon D/{\sim}\to T$ such that $\mu=\mu'\pi$. 
\end{lemma}
\begin{proof} Fix any idempotent element $e_0$ of $D$. Let $x,y\in D$ be such that $x\sim y$. Then $x\vdash e_0=y\vdash e_0$. If $\backslash_\vdash$ is the left inverse operation of $\vdash$, then $x=(x\vdash e_0)\vdash (x\backslash_\vdash x)$ and $y=(y\vdash e_0)\vdash (y\backslash_\vdash y)$. Moreover, $x\backslash_\vdash x, y\backslash_\vdash y\in \U(D)$, so that $\mu(x\backslash_\vdash x)=\mu(y\backslash_\vdash y)$. Therefore $\mu(x)=\mu(y)$. This proves that $\sim$ is contained in the kernel of $\mu$.

From this it follows that $\mu\colon D\to T$ induces a unique digroup morphism $\mu'\colon D/{\sim}\to T$, i.e., there is a unique morphism $\mu'\colon D/{\sim}\to T$ such that 
	$\mu=\mu'\pi$. \end{proof}

We will now show that the classes $\mathsf{T}$ and $\mathsf{F}$ constitute a pretorsion theory in $\Dg$: for every digroup $D$  there is a short $\mathsf{Z}$-preexact sequence $$\xymatrix{
	\U(D) \ar[r]^i &  D \ar[r]^\pi &  D/{\sim}}$$ with $\U(D)\in\mathsf{T}$ and $D/{\sim}\in\mathsf{F}$.
    
\begin{theorem}\label{pretorsion}
The pair $(\mathsf{T},\mathsf{F})$ is a pretorsion theory in $\Dg$. 
\end{theorem}

\begin{proof}
By definition, $\mathsf{Z}:=\mathsf{T}\cap\mathsf{F}$ consists of all digroups with exactly one element. Consider a set $X$ and a group $(G,\cdot)$, so that we have the digroups $(X,\pi_2,\pi_1)$ and $(G,\cdot,\cdot)$. Let $\varphi\colon (X,\pi_2,\pi_1)\to (G,\cdot,\cdot)$ be any morphism in $\Dg$. Given any element $x\in X$ we have that $x\vdash x=x$ and thus $\varphi(x)\cdot \varphi(x)=\varphi(x)$. Since $G$ is a group, we have $\varphi(x)=1_G$, where $1_G$ is the identity element of $G$. Thus $\varphi$ is trivial.

Now let $(D,\vdash,\dashv)$ be any digroup, and consider the group $D/{\sim}$ and  the digroup canonical morphism $\pi\colon D\to D/{\sim}$. Since $\pi(\U(D,\vdash,\dashv))=\{1_{D/{\sim}}\}$, we get from Lemma \ref{Z-prekernel-charact} that the inclusion $i\colon \U(D,\vdash,\dashv)\to D$ is a $\mathsf{Z}$-prekernel of $\pi$. It remains to show that $\pi$ is a $\mathsf{Z}$-precokernel of $i$. Let $\mu:D\to T$ be any morphism in $\Dg$ such that $\mu i$ is trivial. This means that $\mu(\U(D,\vdash,\dashv))$ is a singleton. If we show that there exists a digroup morphism $$\mu'\colon D/{\sim}\to (T,\vdash,\dashv)$$ such that $\mu=\mu'\pi$, then such a $\mu'$ is necessarily unique.
This is an immediate consequence of Lemma~\ref{pretorsion-ruciale}. \end{proof}

\section{Diheaps}\label{7}

In the paper \cite[Section 2]{AFMS2}, it was shown that it is possible to associate to any right group $(G,\cdot)$ a right heap $(G,[-,-,-])$, where the ternary operation $[-,-,-]\colon G\times G\times G\to G$ is defined by $[x,y,z]=x\cdot(y\backslash z)$ for every $x,y,z\in G$. Here $\backslash$ denotes the left inverse of the operation $\cdot$, that is, $x\backslash y$ is the unique element of $G$ such that $x\cdot(x\backslash y)=y$. A {\em right heap} is a pair $(H,[-,-,-])$, where $H$ is a set and $[-,-,-]$ is a ternary operation on $H$ that is associative (that is, $[x,y,[z,t,w]]=[[x,y,z],t,w]$ for every $x,y,z,t,w\in H$), {\em left Mal'tsev} (that is, $[x,x,y]=y$ for every $x,y\in H$), and {\em right weakly Mal'tsev} (i.e., $[x,y,[y,z,w]]=[x,z,w]$ for every $x,y,z,w\in H$). In this way, it is possible to extend the theory of heaps from groups to right groups. It is now obvious that it is possible to give the following definition:

\begin{definition}\label{dh} A {\em diheap} is a triplet $(H,[-,-,-],(-,-,-))$, where $H$ is a set and $[-,-,-],(-,-,-)\colon H\times H\times H\to H$ are two ternary operations on $H$ such that:

{\em (a) } $[x,y,[z,t,w]]=[[x,y,z],t,w]$ \ and\ \ $(x,y,(z,t,w))=((x,y,z),t,w);$

{\em (b) } $[x,x,y]=y$ \ and\ \ $(x,y,y)=x$;

{\em (c) } $[x,y,[y,z,w]]=[x,z,w]$ \ and\ \ $((x,y,z),z,w)=(x,y,w)$;

{\em (d) } $[x,y,(z,t,w)]=([x,y,z],t,w)$  

\noindent for every $x,y,z,t,w\in H$; and

{\em (e) } there exists $e\in H$ such that $[e,x,e]=(e,x,e)$ for every $x\in H$.

\medskip

For a diheap $(H,[-,-,-],(-,-,-))$, we will denote by $$\I(H,[-,-,-],(-,-,-))$$ the non-empty set of all $e\in H$ such that $[e,x,e]=(e,x,e)$ for every $x\in H$.
\end{definition}

An elementary example of a diheap is, for any non-empty set $X$, the diheap $(X,\pi_3,\pi_1)$. Here $\pi_3,\pi_1\colon X\times X\times X\to X$ are the third and the first projection, respectively.

\medskip

The next lemma follows immediately from Definition~\ref{dh}.

\begin{lemma}\label{7.2} {\rm (a)} If $(H,[-,-,-])$ is a non-empty heap, then $$(H,[-,-,-],[-,-,-])$$ is a diheap.

{\rm (b)} If $(H,[-,-,-],(-,-,-))$ is a diheap and $[x,y,z]=(x,y,z)$ for every $x,y,x\in H$, then $(H,[-,-,-])$ is a non-empty heap.

{\rm (c)} If $(H,[-,-,-],(-,-,-))$ is a diheap and $\sim$ is the congruence on $H$ generated by the subset $\{\,([x,y,z],(x,y,z))\mid x,y,z\in H\,\}$ of $H\times H$, then $H/{\sim}$ is a non-empty heap.
    \end{lemma}

Recall that if $(D,\vdash,\dashv)$ is a digroup, then $(D,\vdash)$ is a right group, so that $\vdash$ has a left inverse operation $\backslash_\vdash$, thus  $a\backslash_\vdash b$ is the unique element $t\in D$ such that $a\vdash t=b$. Similarly,  $(D,\dashv)$ is a left group, so that $\dashv$ has a right inverse operation $/_{\!\dashv}$, i.e.,  $a/_{\!\dashv} b$ is the unique element $u\in D$ such that $u\dashv b=a$. 

\begin{theorem}\label{xxx} There is a faithful, essentially surjective functor $F$ from 
the category $\Dg$ of digroups to the category $\DHp$ of diheaps. It associates to every digroup 
    $(D,\vdash,\dashv)$ the diheap $$(D,[-,-,-],(-,-,-)),$$ where $[-,-,-]$ and $(-,-,-)$ are the ternary operations on $D$ defined by $[a,b,c]=a\vdash(b\backslash_\vdash c)$ and $(a,b,c)=(a/_{\!\dashv} b)\dashv c$  for all $a,b,c\in D$. The functor $F$ is the identity on morphisms.
\end{theorem}

\begin{proof} Let $(D,\vdash,\dashv)$ be a digroup.
    For the right group $(D,\vdash)$, we know that $(D,[-,-,-])$ is a right heap by \cite[Theorem~2.8]{AFMS2}. Similarly for the left group $(D,\dashv)$ and its corresponding left heap $(D,(-,-,-))$. Thus properties (a), (b) and (c) of Definition~\ref{dh} are satisfied for the algebra $(D,[-,-,-],$\linebreak $(-,-,-))$. 
    
   As far as (d) is concerned, we have that  $$[x,y,(z,t,w)]=x\vdash(y\backslash_\vdash ((z/_{\!\dashv} t)\dashv w))=x\vdash a,$$ where 
  \begin{equation}
    \label{eq-a}
      y\vdash a=b \dashv w\quad
  \end{equation}
  and
  \begin{equation}
    \label{eq-b}
    b \dashv t=z.\quad
  \end{equation}
  
Similarly, $$([x,y,z],t,w)=((x\vdash(y\backslash_\vdash z))/_{\!\dashv} t)\dashv w=((x\vdash c)/_{\!\dashv} t)\dashv w=d\dashv w,$$  where $c,d\in D$ satisfy the two conditions 
  \begin{equation}
    \label{eq-c}
      y\vdash c=z\quad
  \end{equation}
 and \begin{equation}
    \label{eq-d}
    d\dashv t=x\vdash c.\quad 
  \end{equation}

Hence, to prove (d)  of Definition~\ref{dh}, it suffices to prove that the four equalities \eqref{eq-a}, \eqref{eq-b}, \eqref{eq-c} and \eqref{eq-d} imply that $x\vdash a=d\dashv w$. Now \eqref{eq-b} and \eqref{eq-c} imply that $b \dashv t=y\vdash c$. Also, 
$w=t\vdash f$ for some $f\in D$, and $x=g\dashv y$ for some $g\in D$, so that $x\vdash a=(g\dashv y)\vdash a=(g\vdash y)\vdash a=g\vdash (y\vdash a)=g\vdash (b \dashv w)=(g\vdash b)\dashv w=(g\vdash b)\dashv (t\vdash f)=
(g\vdash b)\dashv t\dashv f=g\vdash (y\vdash c)\dashv f=((g\vdash y)\vdash c)\dashv f=((g\dashv y)\vdash c)\dashv f=(x\vdash c)\dashv f=(d\dashv t)\dashv f=d\dashv w$. This concludes the proof of (d).

            For (e), fix an element $e\in \I(D,\vdash,\dashv)$. Then, for every $x\in D$, we have that $[e,x,e]=e\vdash(x\backslash_\vdash e)=x\backslash_\vdash e$ is the right inverse of $x$ with respect to the operation $\vdash$ relatively to the left identity $e$. Similarly $(e,x,e)$  is the left inverse of $x$ with respect to the operation $\dashv$ relatively to the right identity $e$. But $e\in \I(D,\vdash,\dashv)$ implies that the two inverses coincide, for every $x\in D$. Therefore $[e,x,e]=(e,x,e)$ for every $x\in D$.

    It is now clear that  $F$, that associates to each digroup $(D,\vdash,\dashv)$ the diheap $(D,[-,-,-],(-,-,-))$ and is the identity on morphisms, is a functor $\Dg\to\DHp$, which is obviously faithful.

    In order to prove that $F$ is essentially surjective, fix a diheap $H$ and an element $e\in \I(H)$. Define two binary operations $\vdash_e$ and $\dashv_e$ on $H$ setting, for every $x,y\in H$, $x\vdash_e y=[x,e,y]$ and $x\dashv_e y=(x,e,y)$. From \cite[Theorem~2.8]{AFMS2} we know that $(D,\vdash_e)$ is a right group, $(D,\dashv_e)$ is a left group, and $e\in\E(D,\vdash_e)\cap\E(D,\dashv_e)$. Condition (d) in Definition~\ref{dh} implies the compatibility between $\vdash_e$ and $\dashv_e$, that is condition (b) in Definition~\ref{disemigroup}.

    For (e), we have that $[e,x,e]$ is the right inverse of $x$ with respect to the operation $\vdash_e$ relatively to the left identity $e$, because $x\vdash_e [e,x,e]=[x,e,[e,x,e]]=[x,x,e]=e$. Similarly $(e,x,e)$  is the left inverse of $x$ with respect to the operation $\dashv_e$ relatively to the right identity $e$. Thus condition (e) in Definition~\ref{dh} implies that $e\in I(H,\vdash_e,\dashv_e)$.
    
    The proof that $F(D,\vdash_e,\dashv_e)=(D,[-,-,-](-,-,-))$ is the same as that  in \cite[proof of Theorem~2.8]{AFMS2}.
\end{proof}

We will now elaborate on Theorem~\ref{xxx} and its proof.

As in \cite[Proposition 2.10]{AFMS2}, if we fix a pointed set $(A,a_0)$, then there is a bijection between the set $$\{\,(\vdash,\dashv)\,\mid (A,\vdash,\dashv)\ \mbox{\rm is a digroup and}\  a_0\in\I(A,\vdash,\dashv)\,\}$$ of all pairs $(\vdash,\dashv)$ of binary operations $A$ such that $(A,\vdash,\dashv)$ turns out to be a digroup in which $a_0$ is central idempotent, and the set\linebreak $\{\,([-,-,-],(-,-,-))\,\mid (A,[-,-,-],(-,-,-))$ is a diheap and\linebreak $a_0\in\I(A,[-,-,-],(-,-,-))\,\}$ of all  pairs $([-,-,-],(-,-,-))$ of ternary operations on $A$ for which $(A,[-,-,-],(-,-,-))$ is a diheap and $a_0\in\I(A,[-,-,-],(-,-,-))\,\}$. The proof of this is essentially the same as that of Theorem~\ref{xxx}. If $(\vdash,\dashv)$ and $([-,-,-],(-,-,-))$ correspond in this bijection, that is, $x\vdash y=[x,a_0,y]$ and $x\dashv y=(x,a_0,y)$ for every $x,y\in A$, then:

(1) An equivalence relation on $A$ is a congruence for the digroup $(A,\vdash,\dashv)$, that is, is compatible with  the four operations $\vdash,\ \dashv,\ \backslash_\vdash$ and $ /_\dashv$, if and only if it is compatible with both ternary operations $[-,-,-]$ and $(-,-,-)$.

    (2) A subset of $A$ containing $a_0$ is closed for the four operations $\vdash,\ \dashv,\ \backslash_\vdash$ and $ /_\dashv$ if and only if it  is closed for both operations $[-,-,-]$ and $(-,-,-)$.
      
      (3) Let $f\colon A\to A$ be an idempotent mapping. If $f$ is a digroup endomorphism of the digroup $(A,\vdash,\dashv)$, then $f$ is a diheap endomorphism of $(A,[-,-,-],(-,-,-))$. Conversely, if $f$ is a right heap endomorphism of $(A,[-,-,-],(-,-,-))$ and $f(a_0)=a_0$, then $f$ is a digroup endomorphism of the digroup $(A,\vdash,\dashv)$.

      The proof of (1)--(3) follows immediately from \cite[Proposition 2.10]{AFMS2}.

\bigskip

Recall \cite[Proposition 2.11]{AFMS2} that for any element $a$ in a right heap $H$, there are two idempotent right heap endomorphisms $p_a,q_a\colon H\to H$, defined by $p_a(x)=[x,a,a]$ and $q_a(x)=[a,x,x]$ respectively, for every $x\in H$. The images of $p_a$ and $q_a$ are denoted by $G_a(H)$ and $E_a(H)$ respectively. Also $G_a(H)$ turns out to be a heap, and 
  $E_a(H)$ is a right zero heap, that is, a right heap in which $[x,y,z]=z$ for every $x,y,z\in E_a(H)$.

\begin{lemma} If $(H,[-,-,-], (-,-,-))$ is a diheap, then $$\I(H,[-,-,-],(-,-,-))=\{\, e\in H\mid [x,e,e]=(e,e,x)\ \mbox{\rm for every }x\in H\,\}.$$ Moreover, for every $a\in H$, $$\I(H,\vdash_a,\dashv_a)=\I(H,[-,-,-], (-,-,-))\cap E_a(H).$$\end{lemma}

\begin{proof} For any element $e\in H$ we have that $e\in \I(H,[-,-,-],(-,-,-))$ if and only if $[e,x,e]=(e,x,e)$ for every $x\in H$, that is, if and only if the right/left inverses of $x$ with respect to $\vdash_e$ and $\dashv_e$ relatively to the left/right identity $e$ coincide. Here $a\vdash_e b=[a,e,b]$ and $a\dashv_e b=(a,e,b)$. Now we always have that $[e,e,e]=(e,e,e)=e$, i.e., $e\in \U(H,\vdash_e,\dashv_e)$. Therefore $e\in \I(H,[-,-,-],(-,-,-))$ if and only if 
$e\in \I(H,\vdash_e,\dashv_e)$, if and only if  $x\vdash_e e=e\dashv_e x$ for every $x\in H$ (Proposition~\ref{3.9}(d)). Now $e=e\backslash_{\vdash_e} e$, because $e=e \vdash_e e$, so that $[x,e,e]=x\vdash_e(e\backslash_{\vdash_e} e)=x\vdash_e e$. Similarly $(e,e,x)=e\dashv_e x$. Therefore $e\in \I(H,[-,-,-],(-,-,-))$ if and only if $[x,e,e]=(e,e,x)$ for every $x\in H$.

Now fix an element $a\in H$. Then $\I(H,\vdash_a,\dashv_a)$ is the set of all elements $x\in H$ such that $x\vdash_ax=x$ and every $t\in H$ has a simultaneous inverse with respect to $\vdash_a$ and $\dashv_a$ relatively to $x$. That is, $\I(H,\vdash_a,\dashv_a)$ is the set of all elements $x\in H$ such that $[x,a,x]=x$ and $[x,t,x]=(x,t,x)$ for every $t\in H$. Since $[x,a,x]=x$ if and only if $[a,x,x]=x$, if and only if $q_a(x)=x$, if and only if $x\in E_a$, we get that $$\I(H,\vdash_a,\dashv_a)=\I(H,[-,-,-], (-,-,-))\cap E_a(H).$$ \end{proof}

       A {\em pointed diheap} is a 4-tuple $$(H,[-,-,-], (-,-,-),h),$$ where $(H,[-,-,-], (-,-,-))$ is a diheap and $h$ is a fixed element of\linebreak $\I(H,[-,-,-], (-,-,-))$. Morphisms of pointed diheaps $$(H,[-,-,-], (-,-,-)),h)\to (H',[-,-,-], (-,-,-)),h')$$ are the diheap morphisms, that is, the mappings that respect both ternary operations $[-,-,-]$ and $ (-,-,-)$, that map $h$ to $h'$.
A {\em pointed digroup} is a 4-tuple  $(B,\vdash,\dashv,e)$, where $(B,\vdash,\dashv)$ is a digroup and $e$ is a fixed element of $\I(B,\vdash,\dashv,e)$. Morphisms of pointed right groups $(B,\vdash,\dashv,e)\to (B',\vdash,\dashv,e')$ are the digroup morphisms that map $e$ to $e'$.

\bigskip

Similarly to \cite[Theorem~2.7]{AFMS2}, we have that there is a category isomorphism $G$ of 
the category $\Dg_*$ of pointed digroups to the category $\DHp_*$ of pointed diheaps. It associates to every pointed digroup 
    $(D,\vdash,\dashv,e)$ the pointed diheap $(D,[-,-,-],(-,-,-),e)$, where $[-,-,-]$  and $(-,-,-)$ are the ternary operations on $D$ defined by $$[a,b,c]=a\vdash(b\backslash_{\vdash} c)\quad\mbox{\rm and}\quad (a,b,c)=(a/_{\dashv(} b)\dashv c$$ for all $a,b,c\in D$. The category isomorphism $G$ is the identity on morphisms.
This functor $G\colon \Dg_*\to\DHp_*$  and the functor $F\colon \Dg\to\DHp$ of Theorem~\ref{xxx} are related by the commutative square $$\xymatrix{
{\Dg_*}\ar[r]^{G}\ar[d] &{\DHp_*}\ar[d] \\
{\Dg}\ar[r]_F & \DHp, }$$ where the vertical arrows are the forgetful functors that forget the base points.

\begin{proposition} Let $H$ be a diheap. For every $e\in \I(H)$, the mapping $p_e\colon H\to H$, defined by $p_e(x)=[x,e,e]$ for every $x\in H$, is an idempotent diheap endomorphism of $H$. Its kernel is the congruence on $H$, that is, the equivalence relation on $H$ compatible with both ternary operations $[-,-,-]$ and $ (-,-,-)$, generated by the subset $\{\,([x,y,z],(x,y,z))\mid x,y,z\in H\,\}$ of $H\times H$. \end{proposition}

\begin{proof} The mapping $p_e\colon H\to H$ is defined by $p_e(x)=[x,e,e]=(e,e,x)$ for every $x\in H$, and is an idempotent diheap endomorphism of $H$ by \cite[Proposition~2.11]{AFMS2} applied to the right heap $(H, [-,-,-])$, and by its right/left symmetric applied to the left heap $(H,(-,-,-))$. 

Let us show that the kernel of $p_e$ is the congruence $\gamma$ generated by the subset $\{\,([x,y,z],(x,y,z))\mid x,y,z\in H\,\}$ of $H\times H$. 

In order to show that the kernel of $p_e$ contains $\gamma$,  it suffices to show that $p_e([x,y,z])=p_e((x,y,z))$ for every $x,y,z\in H$.  Now, given the diheap $H$ and any element $e\in\I(H)$, so that $[x,e,e]=(e,e,x)$ for every $x\in H$, we can construct a digroup $(H,\vdash_e,\dashv_e)$ setting $x\vdash_e y=[x,e,y]$ and  $x\dashv_e y=(x,e,y)$ for every $x,y\in H$. With respect to these operations, $p_e\colon H\to H$ turns out an idempotent digroup endomorphism whose image $p_e(H)$ is a group, so that $\vdash_e$ and $\dashv_e$ coincide on  $p_e(H)$. Moreover $e=[e,e,e]=p_e(e)$ is the identity of the group $p_e(H)$. On the group $(p_e(H),\vdash_e,\dashv_e)=(p_e(H),\cdot,\cdot)$, we have that  $a\backslash_{\vdash_e} b=a^{-1}b$ for all $a,b\in p_e(H)$, so that, on $p_e(H)$, $[a,b,c]=a(b^{-1}c)$. It follows that $p_e([x,y,z])=[p_e(x),p_e(y),p_e(z)]=p_e(x)(p_e(y))^{-1}p_e(z)$. Similarly, on the group $p_e(H)$, we get that $a/_{\!\dashv_e} b=ab^{-1}$ for every $a,b\in p_e(H)$, so that, on $p_e(H)$, $(a,b,c)=(ab^{-1})c$. It follows that $$p_e((x,y,z))=(p_e(x),p_e(y),p_e(z))=p_e(x)(p_e(y))^{-1}p_e(z)=p_e([x,y,z]).$$ This shows that the kernel of $p_e$ contains $\gamma$.

To prove that the kernel of $p_e$ is contained in $\gamma$, we must show that, for every $x,y\in H$ with $p_e(x)=p_e(y)$, we have that $x\,\gamma\, y$. Here $\gamma$ is the congruence generated by  the set $\{\,([x,y,z],(x,y,z))\mid x,y,z\in H\,\}$. Now, on the quotient digroup $H/{\gamma}$, the two ternary operations $[-,-,-]$ and $(-,-,-)$ coincide, so that $H/{\gamma}$ is a heap. Hence in $H$ we see that $x\,\gamma\,[x,e,e]=p_e(x)=p_e(y)=[y,e,e]\,\gamma\, y$. Hence $x\,\gamma\, y$, as desired.
\end{proof}

\end{document}